\documentclass[12pt,reqno]{article}

\usepackage[usenames]{color}
\usepackage{amssymb}
\usepackage{graphicx}
\usepackage{amscd}

\usepackage[colorlinks=true,
linkcolor=webgreen,
filecolor=webbrown,
citecolor=webgreen]{hyperref}

\definecolor{webgreen}{rgb}{0,.5,0}
\definecolor{webbrown}{rgb}{.6,0,0}

\usepackage{color}
\usepackage{fullpage}
\usepackage{float}

\usepackage{graphics,amsmath,amssymb}
\usepackage{amsthm}
\usepackage{amsfonts}
\usepackage{latexsym}
\usepackage{epsf}

\newcommand{\seqnum}[1]{\href{http://oeis.org/#1}{\underline{#1}}}

\newcommand{\bta}{010201}
\newcommand{\btb}{010210}
\newcommand{\btc}{010212}
\newcommand{\btd}{012021}
\newcommand{\bte}{012101}
\newcommand{\btf}{012102}
\newcommand{\btg}{020102}
\newcommand{\bth}{020120}
\newcommand{\bti}{020121}
\newcommand{\btj}{021012}
\newcommand{\btk}{021201}
\newcommand{\btl}{021202}
\newcommand{\btm}{101201}
\newcommand{\btn}{101202}
\newcommand{\bto}{102012}
\newcommand{\btp}{120102}
\newcommand{\btq}{120121}
\newcommand{\btr}{120212}
\newcommand{\bts}{121012}
\newcommand{\btt}{201202}
\newcommand{\btu}{210212}
\newcommand{\obta}{102010}
\newcommand{\obtb}{012010}
\newcommand{\obtc}{212010}
\newcommand{\obtd}{120210}
\newcommand{\obte}{101210}
\newcommand{\obtf}{201210}
\newcommand{\obtg}{201020}
\newcommand{\obth}{021020}
\newcommand{\obti}{121020}
\newcommand{\obtj}{210120}
\newcommand{\obtk}{102120}
\newcommand{\obtl}{202120}
\newcommand{\obtm}{102101}
\newcommand{\obtn}{202101}
\newcommand{\obto}{210201}
\newcommand{\obtp}{201021}
\newcommand{\obtq}{121021}
\newcommand{\obtr}{212021}
\newcommand{\obts}{210121}
\newcommand{\obtt}{202102}
\newcommand{\obtu}{212012}

\begin{document}

\theoremstyle{plain}
\newtheorem{theorem}{Theorem}
\newtheorem{corollary}[theorem]{Corollary}
\newtheorem{lemma}[theorem]{Lemma}
\newtheorem{proposition}[theorem]{Proposition}
\theoremstyle{definition}
\newtheorem{definition}[theorem]{Definition}
\newtheorem{example}[theorem]{Example}
\newtheorem{conjecture}[theorem]{Conjecture}
\theoremstyle{remark}
\newtheorem{remark}[theorem]{Remark}
\newtheorem{step}{Step}

\begin{center}
\epsfxsize=4in
\end{center}

\begin{center}
\vskip 1cm{\LARGE\bf An Improved Lower Bound for $n$-Brinkhuis $k$-Triples}
\vskip 1cm
Michael Sollami\\
Ditto Labs\\
Cambridge, MA, U.S.A.\\
\href{mailto:michaelsollami@gmail.com}{\tt michaelsollami@gmail.com} \\
\ \\
Craig C. Douglas\\
University of Wyoming\\
School of Energy Resources and Department of Mathematics\\
1000 E. University Ave., Dept. 3036\\
Laramie, WY 82072, U.S.A.\\
\href{mailto:cdougla6@uwyo.edu}{\tt cdougla6@uwyo.edu}\\
\ \\

Manfred Liebmann\\
Technische Universit\"{a}t M\"{u}nchen\\
Center for Mathematical Sciences\\
Boltzmannstra{\ss}e 3\\
85748 Garching by Munich, Germany\\
\href{mailto:manfred.liebmann@tum.de}{\tt manfred.liebmann@tum.de}
\end{center}

\begin{abstract}
Let $s_n$ be the number of words in the ternary alphabet $\Sigma = \{0, 1, 2\}$ such 
that no subword (or factor) is a square (a word concatenated with itself, e.g., $11$, $1212$, or $102102$).
From computational evidence, $s_n$ grows exponentially at a rate of about
$1.317277^n$. While known upper bounds are already relatively close to the conjectured rate,
effective lower bounds are much more difficult to obtain.
In this paper, we construct a $54$-Brinkhuis $952$-triple, which leads to an
improved lower bound on the number of $n$-letter ternary squarefree words:
\mbox{$952^{n/53} \approx 1.1381531^n$}.
\end{abstract}

\section{Introduction}
\label{Sec:Introduction}

A {\em word} $w$ of length $n$ is a string of $n$ symbols from an alphabet
$\Sigma$.
A word $w$ is said to be {\em squarefree} if it does not contain an adjacent
repetition of a   
{\em subword} (or {\em factor}), i.e., $w$ cannot be written
as $axxb$ for nonempty subwords $a$, $x$, and $b$.
In the field of combinatorics on words, the literature on pattern-avoiding
words is vast and there has always been much progress in the study of
powerfree words such as the {\em binary cubefree} and {\em ternary squarefree}
words (see \cite{Finch_URL, WolframResearch_URL}).

It is easy to see that there are only six nonempty binary squarefree words:
$\{0,1,01,10,101,010\}$.
Using the Prouhet-Thue-Morse sequence (see \cite{Lothaire_1983}) the number of
ternary squarefree words was proven to be infinite.

We denote by $s_n$ the exponentially growing number \cite{Richard_Grimm_2004}
of ternary squarefree words of length $n$ 
\cite{Allouche_Shallit_2003,Lothaire_1983}.
We denote by $\mathcal{A}(n)$ the set of ternary squarefree words of length
$n$.

In Section~\ref{Sec:BrinkhuisTriples}, we define basic properties of
$n$-Brinkhuis $k$-triples.
In Section~\ref{Sec:Searching4BrinkhuisTriples}, we define how we searched
for the $54$-Brinkhuis $952$-triple.
In Section~\ref{Sec:54-Brinkhuis_952-triple}, we produce the newly
discovered $54$-Brinkhuis $952$-triple.
In Section~\ref{Sec:CodeAvailability}, we describe how to get the code that
found the specialized Brinkhuis triple of
Section~\ref{Sec:54-Brinkhuis_952-triple} and how to run it.

\section{$n$-Brinkhuis $k$-triples}
\label{Sec:BrinkhuisTriples}

In this section, we define a $n$-Brinkhuis $k$-triple, prove a theorem
about the lower bound on the growth rate $s$, and provide a history of
estimates for the lower bound.

\begin{definition}
An $n$-Brinkhuis $k$-triple is a set
$\mathcal{B}$ = \{$\mathcal{B}^0$, $\mathcal{B}^1$, $\mathcal{B}^2$\}
of three sets of words 
$\mathcal{B}^i = \{w^i_j ~|~ 1 \le j \le k\}$.
The $w^i_j$ are squarefree words of length $n$ such that for all squarefree
words $i_1i_2i_3$ with
$i_1,i_2,i_3\in\{0,1,2\}$
has the property that
the word $w^{i_1}_{j_1}w^{i_2}_{j_2}w^{i_3}_{j_3}$ of length $3n$ with
$j_1,j_2,j_3\in\{1,2,\cdots,k\}$
is also squarefree.
\end{definition}

An example of an $18$-Brinkhuis $2$-triple \cite{Ekhad_Zeilberger_1998} is
given by,
\[
\begin{array}{lccllc}
\mathcal{B} & = & \{ & \mathcal{B}^0 = \{210201202120102012, & 210201021202102012\}, & \\
            &   &    & \mathcal{B}^1 = \{021012010201210120, & 021012102010210120\}, & \\
            &   &    & \mathcal{B}^2 = \{102120121012021201, & 102120210121021201\}  & \}.
\end{array}
\]

The lower bound on the growth rate is given in the following theorem
\cite{Brinkhuis_1983}:

\begin{theorem} 
\label{Thm:BTLowerBound}
The existence of a special $n$-Brinkhuis $k$-triple implies that the
lower bound on the growth rate of the ternary squarefree words is
\[
k^{1/(n-1)}\leq s = \underset{m\rightarrow\infty}{\lim} (s_{m})^{1/m}.
\]
\end{theorem}
\begin{proof}
We define the a set of uniformly growing morphisms by
\begin{displaymath}
   \rho : \left\{
     \begin{array}{lr}
       0\rightarrow w_{j_0}^0, & \\[0.15cm]
       1\rightarrow w_{j_1}^1, & \\[0.15cm]
       2\rightarrow w_{j_2}^2, &       
     \end{array}
   \right.
\end{displaymath} 
where $1\leq j_0,j_1,j_2 \leq k$.
As proven in \cite{Brandenburg_1983,Crochemore_1982,Leconte_1985}, the $\rho$
are squarefree morphisms mapping each squarefree word of length $m$ to $k^m$
squarefree words of length $nm$.
Thus, existence of an $n$-Brinkhuis $k$-triple indicates that
\[s_{mn}/s_{m} \geq k^m,\,\, \forall\, m,n\geq 1.\]
Since $s=\underset{m\rightarrow\infty}{\lim} (s_{m})^{1/m}$,
\[s^{n-1}= \underset{m\rightarrow\infty}{\lim} \left(\frac{s_{mn}}{s_{m}}\right)^{1/m} \geq k,\]
which yields the lower bound of $s \geq k^{1/(n-1)}$. 
\end{proof}

A history of estimates for the lower bound for $s$ is given in
Table~\ref{Tbl:LowerBounds}.
As is obvious from Theorem~\ref{Thm:BTLowerBound}, the discovery of a
$n$-Brinkhuis $k$-triple for a new pair ($n$,$k$) potentially gives us a new a
lower bound for $s$.

\begin{table}[htbp]
\begin{center}
\begin{tabular}{| c | c | c | c | l |}\hline
$n$ & $k$ & Lower bound                  & Year & Authors\\ \hline
25  & 2   & $2^{n/24}\approx 1.0293022^n$   & 1983 & Brinkhuis \cite{Brinkhuis_1983}\\
22  & 2   & $2^{n/21}\approx 1.0335578^n$   & 1983 & Brandenburg \cite{Brandenburg_1983}\\
18  & 2   & $2^{n/17}\approx 1.0416160^n$   & 1998 & Ekhad and Zeilberger \cite{Ekhad_Zeilberger_1998}\\
41  & 65  & $65^{n/40}\approx 1.1099996^n$  & 2001 & Grimm \cite{Grimm_2001}\\
43  & 110 & $110^{n/42}\approx 1.1184191^n$ & 2003 & Sun \cite{Sun_2003}\\
54  & 952 & $952^{n/53}\approx 1.1381531^n$ & 2016 & Sollami, Douglas, and Liebmann\\
\hline
\end{tabular}
\end{center}
\caption{Lower bounds.}
\label{Tbl:LowerBounds}
\end{table}

\section{Searching for $n$-Brinkhuis $k$-triples}
\label{Sec:Searching4BrinkhuisTriples}

In this section we describe how we searched for $n$-Brinkhuis $k$-triples.

We can pare down the search by systematically determining the prefixes and
suffixes of the words in a special $n$-Brinkhuis $k$-triple.
Grimm \cite{Grimm_2001} proved that only two classes of special $n$-Brinkhuis
$k$-triples must be searched, namely
\[
\mathcal{A}_1(n)=\{w\in \mathcal{A}(n) ~|~ w=012021\{012\}*120210\}\subseteq
\mathcal{A}(n)
\]
and
\[
\mathcal{A}_2(n)=\{w\in \mathcal{A}(n) ~|~ w=012102\{012\}*201210\}\subseteq
\mathcal{A}(n),
\]
where recall that $\mathcal{A}(n)$ is the set of ternary squarefree words of
length $n$.

Let $\overline{w}$ be the reversal of symbols in $w$.
We use this notation to be consistent with earlier papers in this field, e.g.,
\cite{Grimm_2001,Sun_2003} (it is sometimes denoted by $w^R$ by other
authors).
For example, if
\[
w = 0122,~ \overline{w} = 2210
\]
and an example palindrome is
\[
w = 2112 = \overline{w}.
\]

We denote the number of potential words, palindromes, and nonpalindromes for
each set $\mathcal{A}_i(n)$, $i\in\{1,2\}$, by
\begin{align*}
a_i(n) & =|\mathcal{A}_i(n)|,\\
a_{ip}(n) & = |\{w \in \mathcal{A}_i(n) ~|~ w = \overline{w}\}|,\\
a_{in}(n) & = |\{w \in \mathcal{A}_i(n) ~|~ w \neq \overline{w}\}|.
\end{align*}
Clearly, there are no palindromic squarefree words of even length.
Thus,
$a_{1p}(2n) = a_{2p}(2n) = 0$,
$a_{1n}(2n) = a_1(2n)/2$, and
$a_{2n}(2n) = a_2(2n)/2$ \cite{Grimm_2001}.
If a word in $\mathcal{A}_1(n)$ or $\mathcal{A}_2(n)$ is a member of a
special $n$-Brinkhuis $k$-triple, then it must at least generate a Brinkhuis
triple by itself, which motivates the following definition:

\begin{definition}
A word $w$ is {\em admissible} if $\{w,\tau(w),\tau^2 (w)\}$ is a special
$n$-Brinkhuis $k$-triple, where $\tau$ is the permutation
\begin{equation}
\label{equ:tau}
\tau : \left\{
\begin{array}{ccc}
0 & \rightarrow & 1,\\
1 & \rightarrow & 2,\\
2 & \rightarrow & 0.\\
\end{array}\right.
\end{equation}
\end{definition}

As before, we denote the number of admissible words, palindromes, and
nonpalindromes for each set $\mathcal{A}_i(n)$, $i\in\{1,2\}$, by
\begin{align*}
b_i(n) &=|\{w\in \mathcal{A}_i(n) \text{ and $w$ is admissible}\}|,\\
b_{ip}(n) &= |\{w \in \mathcal{A}_i(n) ~|~ w = \overline{w} \text{ and $w$ is admissible}\}|,\\
b_{in}(n) &= |\{w \in \mathcal{A}_i(n) ~|~ w \neq \overline{w} \text{ and $w$ is admissible}\}|.
\end{align*}

The strategy we used to find a special $n$-Brinkhuis $k$-triple begins by
enumerating the set of all admissible words of length $n$.
From this enumeration we determine the largest subset in which any three words
$w_1$, $w_2$, $w_3$, form a special $n$-Brinkhuis triple.

The method we used to find a special $n$-Brinkhuis $k$-triple is summarized
below in three steps:

\begin{step}
\label{Step:all}
Find all admissible words in $\mathcal{A}_1(n)$ and $\mathcal{A}_2(n)$.
\end{step}
\begin{step}
\label{Step:Brinkhuis}
Find all triples of admissible words that generate a special $n$-Brinkhuis
$k$-triple.
\end{step}
\begin{step}
\label{Step:largest}
Find the largest set of admissible words such that all three-elemental subsets
are contained in our list of admissible triples.
\end{step}

Steps~\ref{Step:all} and \ref{Step:Brinkhuis} are essentially precomputations
which involve checking the squarefreeness of words.
A naive algorithm for detecting squares has time complexity of order $O(n^3)$
for words of length $n$ and a fixed length alphabet.
This algorithm was first improved to of order $O(n\, \log\,n)$
\cite{Crochemore_1982} and then further improved to of order $O(n)$
\cite{Allouche_Shallit_2003}.

Experimentally it seems that $\mathcal{A}_1$ is more likely to provide maximum
sized $n$-Brinkhuis triples for large $n$ than generators from the set
$\mathcal{A}_2$ and so we have focused our search to this specific
class~\cite{Grimm_2001}.
It is also simpler to find $n$-Brinkhuis $k$-triples in $\mathcal{A}_1(n)$
where $n$ is even since a maximum number of generators does not necessarily
give the largest Brinkhuis triple (i.e., unless we know that none of the words
are palindromes, as they are for even $n$).

It is in the Step~\ref{Step:largest} that our main difficulty becomes
apparent.
The way we found $n$-Brinkhuis $k$-triples involved solving a purely
combinatorial problem that is an instance of the NP-complete maximum clique
problem for hypergraphs \cite{RotaBulo_Pelillo_2008}.
A maximal hyperclique in a hypergraph on $n$ vertices with hyperedges of
cardinality at most $\aleph$ can be found using a branching algorithm in
$\mathcal{O}(2^{\kappa n})$ time for some $\kappa < 1$, depending only on
$\aleph$ \cite{Bodlaender_Italiano_2013}.

\section{A $54$-Brinkhuis $952$-triple}
\label{Sec:54-Brinkhuis_952-triple}

\begin{theorem} A special $54$-Brinkhuis $952$-triple exists, and thus shows 
\[
s \geq 952^{1/53} \approx 1.1381531 > 110^{1/42} \approx 1.1184191.
\]
\end{theorem}
\begin{proof}
The proof is by a computational construction of a special Brinkhuis triple.
$\mathcal{B}^0$ is explicitly listed below.
$\mathcal{B}^1$ is constructed by applying the $\tau$ permutation 
\ref{equ:tau} on $\mathcal{B}^0$.
$\mathcal{B}^2$ is constructed by applying the $\tau$ permutation on
$\mathcal{B}^1$.
All three sets are available as plain ASCII files on the journal web site
and \cite{DouglasURL}.

Practical algorithms have been developed to solve the maximum clique problem
(see \cite{Bomze_Budinich_Pardalos_Pelillo_1999} for a comprehensive survey on
methods of finding maximum cliques).
These methods were adapted to solve the corresponding problem for hypergraphs.
We used the {\em Random Hyperclique Search algorithm} (RHCS) to perform our
computer searches for maximum hyper-cliques \cite{Sollami_2013}.

Formally, the first $476$ elements of $\mathcal{B}^0$ are given below.
The remaining $476$ elements are reversals of the first $476$ elements.
\begin{quote}
\noindent
$ \btd \btg \btp \btd \btg \btn \btq \btg \obtd $,\\
$ \btd \btg \btq \obtb \obts \btg \obtd \obti \obtd $,\\
$ \btd \btg \obtd \btt \btp \obte \btt \btp \obtd $,\\
$ \btd \bti \btk \bti \obtb \obtu \btn \btp \obtd $,\\
$ \btd \obtg \obtd \obtp \obtb \obtf \obtp \btf \obtd $,\\
$ \btd \obtg \obtq \obtg \obtd \obtp \obtf \btp \obtd $,\\
$ \btd \obtp \btf \bte \obtt \btd \obtf \btp \obtd$,\\ ~\\
$ \btd \btg \btp \btd \btg \obte \obtp \obtg \obtd $,\\
$ \btd \btg \btq \obtb \obts \bth \obts \btg \obtd $,\\
$ \btd \btg \obtd \btt \btp \obte \obtf \btp \obtd $,\\
$ \btd \bti \btk \bti \btd \btf \btd \btg \obtd $,\\
$ \btd \obtg \obtd \obtp \obtb \obtu \btn \btp \obtd $,\\
$ \btd \obtg \obtq \obtg \btr \bte \obtp \btf \obtd $,\\
$ \btd \obtp \btf \btp \btd \btg \btn \btp \obtd $,\\ ~\\
$ \btd \btg \btp \btd \btg \obte \btt \btp \obtd $,\\
$ \btd \btg \btq \obtb \obts \btl \btn \btp \obtd $,\\
$ \btd \btg \obtd \btt \btp \obtd \obtp \btf \obtd $,\\
$ \btd \bti \btk \bti \btd \btf \btq \btg \obtd $,\\
$ \btd \obtg \obtd \obtp \btd \bth \obts \btg \obtd $,\\
$ \btd \obtg \obtq \obtg \btr \btf \btd \btf \obtd $,\\
$ \btd \obtp \btf \btp \btd \obtf \obtp \btf \obtd $,\\ ~\\
$ \btd \btg \btp \btd \btg \obte \obtf \btp \obtd $,\\
$ \btd \btg \btq \obtb \obtr \btf \btb \btp \obtd $,\\
$ \btd \btg \obtd \btt \btq \bth \obts \btg \obtd $,\\
$ \btd \bti \btk \bti \btd \btg \btq \btg \obtd $,\\
$ \btd \obtg \obtd \obtp \btf \btp \btd \btf \obtd $,\\
$ \btd \obtg \obtq \obtp \obtb \btt \btq \btg \obtd $,\\
$ \btd \obtp \btf \btp \btf \btq \btd \btf \obtd $,\\ ~\\
$ \btd \btg \btp \btd \bti \btd \obtf \btp \obtd $,\\
$ \btd \btg \btq \obtb \obtr \btf \btd \obtg \obtd $,\\
$ \btd \btg \obtd \obtf \btp \obte \btt \btp \obtd $,\\
$ \btd \bti \btk \bti \btd \bth \obts \btg \obtd $,\\
$ \btd \obtg \obtd \obtp \btf \btp \btn \btp \obtd $,\\
$ \btd \obtg \obtq \obtp \obtb \obtf \obtp \btf \obtd $,\\
$ \btd \obtp \btf \btp \btf \obtd \obtp \btf \obtd $,\\ ~\\
$ \btd \btg \btp \btd \obtp \btd \obtf \btp \obtd $,\\
$ \btd \btg \btq \obtb \obtr \btf \btq \btg \obtd $,\\
$ \btd \btg \obtd \obtf \btp \obte \obtf \btp \obtd $,\\
$ \btd \bti \btk \bti \btg \btp \btd \btf \obtd $,\\
$ \btd \obtg \obtd \obtp \btf \btq \btd \btf \obtd $,\\
$ \btd \obtg \obtq \obtp \obtb \obtu \btn \btp \obtd $,\\
$ \btd \obtp \btf \btp \btm \bti \btd \btf \obtd $,\\ ~\\
$ \btd \btg \btp \btd \obtp \btf \btb \btp \obtd $,\\
$ \btd \btg \btq \obtb \obtr \btg \btn \btp \obtd $,\\
$ \btd \btg \obtd \obtf \btp \btq \btd \btf \obtd $,\\
$ \btd \bti \btk \bti \btg \btp \btd \btg \obtd $,\\
$ \btd \obtg \obtd \obtp \btf \obtd \btt \btp \obtd $,\\
$ \btd \obtg \obtq \obtp \btd \btg \btn \btp \obtd $,\\
$ \btd \obtp \btf \btp \btm \btl \btn \btp \obtd $,\\ ~\\
$ \btd \btg \btp \btd \obtp \btf \btd \btf \obtd $,\\
$ \btd \btg \btq \obtb \obtr \btg \btq \btg \obtd $,\\
$ \btd \btg \obtd \obtf \btp \btq \btd \btg \obtd $,\\
$ \btd \bti \btk \bti \btg \btq \btd \btf \obtd $,\\
$ \btd \obtg \obtd \obtp \obtg \btr \btb \btp \obtd $,\\
$ \btd \obtg \obtq \obtp \btd \obtf \obtp \btf \obtd $,\\
$ \btd \obtp \btf \btp \btn \btp \btd \btf \obtd $,\\ ~\\
$ \btd \btg \btp \btd \obtp \btf \btd \btg \obtd $,\\
$ \btd \btg \btq \obtb \obtr \bth \obts \btg \obtd $,\\
$ \btd \btg \obtd \obtf \btp \obtd \obtp \btf \obtd $,\\
$ \btd \bti \btk \bti \btg \btq \btd \btg \obtd $,\\
$ \btd \obtg \obtd \obtp \obtg \obtq \obtp \btf \obtd $,\\
$ \btd \obtg \obtq \obtp \btf \btp \btd \btf \obtd $,\\
$ \btd \obtp \btf \btq \obtb \obtf \obtp \btf \obtd $,\\ ~\\
$ \btd \btg \btp \btd \obtp \btf \btd \obtg \obtd $,\\
$ \btd \btg \btq \btd \btg \obte \btt \btp \obtd $,\\
$ \btd \btg \obtd \obtf \btp \obtd \btt \btp \obtd $,\\
$ \btd \bti \btk \bti \btg \obtd \obtp \btf \obtd $,\\
$ \btd \obtg \obtd \obtp \obtg \obtq \obtf \btp \obtd $,\\
$ \btd \obtg \obtq \obtp \btf \btq \btd \btf \obtd $,\\
$ \btd \obtp \btf \btq \btd \btg \btn \btp \obtd$,\\ ~\\
$ \btd \btg \btp \btd \obtp \btf \btq \btg \obtd $,\\
$ \btd \btg \btq \btd \btg \obtd \obtq \obtg \obtd $,\\
$ \btd \btg \obtd \obtf \obtd \obtp \obtf \btp \obtd $,\\
$ \btd \bti \btk \bti \btg \obtd \btt \btp \obtd $,\\
$ \btd \obtg \obtd \obtp \obtf \btp \btd \btf \obtd $,\\
$ \btd \obtg \obtq \obtp \btf \obtd \btt \btp \obtd $,\\
$ \btd \obtp \btf \btq \btd \obtf \obtp \btf \obtd$,\\ ~\\
$ \btd \btg \btp \btd \obtp \btf \obtd \obti \obtd $,\\
$ \btd \btg \btq \btd \btg \obtd \btt \btp \obtd $,\\
$ \btd \btg \obtd \obtf \btr \bte \obtp \btf \obtd $,\\
$ \btd \bti \btk \btj \obth \btr \btb \btp \obtd $,\\
$ \btd \obtg \obtd \obtp \obtf \btp \btn \btp \obtd $,\\
$ \btd \obtg \obtq \obtf \btp \obte \btt \btp \obtd $,\\
$ \btd \obtp \btf \btq \btg \btp \btd \btf \obtd $,\\ ~\\
$ \btd \btg \btp \btd \obtp \obtn \obtl \obti \obtd $,\\
$ \btd \btg \btq \btd \bth \obtj \obtk \obti \obtd $,\\
$ \btd \btg \obtd \obtf \btr \btf \btd \obtg \obtd $,\\
$ \btd \bti \btk \btj \obta \obtu \btn \btp \obtd $,\\
$ \btd \obtg \obtd \obtp \obtf \btp \btq \btg \obtd $,\\
$ \btd \obtg \obtq \obtf \btp \obte \obtf \btp \obtd $,\\
$ \btd \obtp \btf \btq \btg \btq \btd \btf \obtd $,\\ ~\\
$ \btd \btg \btp \btd \obtp \obtt \btb \btp \obtd $,\\
$ \btd \btg \btq \btd \bth \obtu \btn \btp \obtd $,\\
$ \btd \btg \obtd \obtf \obtc \obtf \obtp \btf \obtd $,\\
$ \btd \bti \btk \btj \bto \btj \obtk \obti \obtd $,\\
$ \btd \obtg \obtd \obtp \obtf \obtd \obtp \btf \obtd $,\\
$ \btd \obtg \obtq \obtf \btp \obtd \obtp \btf \obtd $,\\
$ \btd \obtp \btf \btq \btg \obtd \obtp \btf \obtd $,\\ ~\\
$ \btd \btg \btp \btd \obtp \obtt \btd \btf \obtd $,\\
$ \btd \btg \btq \btd \bti \btd \obtf \btp \obtd $,\\
$ \btd \btg \obtd \obtf \obtu \btn \btq \btg \obtd $,\\
$ \btd \bti \btk \btj \bto \btn \btq \btg \obtd $,\\
$ \btd \obtg \obtd \obtp \obtf \obtd \btt \btp \obtd $,\\
$ \btd \obtg \obtq \obtf \btp \obtd \btt \btp \obtd $,\\
$ \btd \obtp \btf \btq \btg \obtd \btt \btp \obtd $,\\ ~\\
$ \btd \btg \btp \btd \obtp \obtt \btd \obtg \obtd $,\\
$ \btd \btg \btq \btd \obtg \obtd \obtp \btf \obtd $,\\
$ \btd \btg \obtd \obtf \obtr \btf \btd \btf \obtd $,\\
$ \btd \bti \btk \btj \obtk \obto \obtl \obti \obtd $,\\
$ \btd \obtg \obtd \obtp \obtf \btr \btb \btp \obtd $,\\
$ \btd \obtg \obtq \obtf \obtd \obtq \obtp \btf \obtd $,\\
$ \btd \obtp \btf \btq \bth \obtu \btn \btp \obtd $,\\ ~\\
$ \btd \btg \btp \btd \obtf \btp \btd \btf \obtd $,\\
$ \btd \btg \btq \btd \obtg \obti \obtk \obti \obtd $,\\
$ \btd \btg \obtd \obtf \obtr \btg \btn \btp \obtd $,\\
$ \btd \bti \btk \obtj \bto \btj \obtk \obti \obtd $,\\
$ \btd \obtg \obtd \obtp \obtf \btt \btq \btg \obtd $,\\
$ \btd \obtg \obtq \obtf \obtd \obtq \obtf \btp \obtd $,\\
$ \btd \obtp \btf \obtd \btp \btq \btd \btf \obtd $,\\ ~\\
$ \btd \btg \btp \btd \obtf \btp \btd \btg \obtd $,\\
$ \btd \btg \btq \btd \obtp \btd \obtf \btp \obtd $,\\
$ \btd \btg \obtd \obtf \obtr \btg \btq \btg \obtd $,\\
$ \btd \bti \btk \obtj \bto \btn \btq \btg \obtd $,\\
$ \btd \obtg \obtd \obtp \obtn \btu \btb \btp \obtd $,\\
$ \btd \obtg \obtq \obtf \obtd \obtp \obtf \btp \obtd $,\\
$ \btd \obtp \btf \obtd \btp \obtd \obtp \btf \obtd $,\\ ~\\
$ \btd \btg \btp \btd \obtf \btp \btn \btp \obtd $,\\
$ \btd \btg \btq \btd \obtp \btf \btd \btf \obtd $,\\
$ \btd \btg \obtd \obtf \obtr \btg \obtd \obti \obtd $,\\
$ \btd \bti \btk \obtj \bto \obtk \obts \btg \obtd $,\\
$ \btd \obtg \obtd \obtp \obtt \btd \obtp \btf \obtd $,\\
$ \btd \obtg \obtq \obtf \obtd \btt \btq \btg \obtd $,\\
$ \btd \obtp \btf \obtd \btp \obtd \btt \btp \obtd $,\\ ~\\
$ \btd \btg \btp \btd \obtf \btp \btq \btg \obtd $,\\
$ \btd \btg \btq \btd \obtp \btf \btd \btg \obtd $,\\
$ \btd \bti \obtb \btt \btm \bti \btd \btf \obtd $,\\
$ \btd \bti \btk \obtj \obtm \obtt \btb \btp \obtd $,\\
$ \btd \obtg \obtd \obtp \obtt \btd \obtf \btp \obtd $,\\
$ \btd \obtg \obtq \obtf \btr \btf \btb \btp \obtd $,\\
$ \btd \obtp \btf \obtd \btr \bte \obtp \btf \obtd $,\\ ~\\
$ \btd \btg \btp \btd \obtf \obtd \obtp \btf \obtd $,\\
$ \btd \btg \btq \btd \obtp \btf \btd \obtg \obtd $,\\
$ \btd \bti \obtb \obtf \obtp \obtf \btt \btp \obtd $,\\
$ \btd \bti \btk \obtj \obtm \obtt \btd \btf \obtd $,\\
$ \btd \obtg \obtd \obtp \obtt \bte \obtp \btf \obtd $,\\
$ \btd \obtg \obtq \obtf \btr \btf \btd \btf \obtd $,\\
$ \btd \obtp \btf \obtd \obti \obtd \obtp \btf \obtd $,\\ ~\\
$ \btd \btg \btp \btd \obtf \obtd \btt \btp \obtd $,\\
$ \btd \btg \btq \btd \obtp \btf \obtd \obti \obtd $,\\
$ \btd \bti \obtb \obtf \obtc \btt \btq \btg \obtd $,\\
$ \btd \bti \btk \obtj \obtm \obto \obtl \obti \obtd $,\\
$ \btd \obtg \obtd \obtf \btp \obte \btt \btp \obtd $,\\
$ \btd \obtg \obtq \obtf \btt \obte \obtp \btf \obtd $,\\
$ \btd \obtp \btf \obtd \btt \obte \obtp \btf \obtd $,\\ ~\\
$ \btd \btg \btp \btd \obtf \btr \btb \btp \obtd $,\\
$ \btd \btg \btq \btd \obtp \obtt \btb \btp \obtd $,\\
$ \btd \bti \obtb \obtf \obtu \btn \btq \btg \obtd $,\\
$ \btd \bti \btk \obtj \obtm \btu \btb \btp \obtd $,\\
$ \btd \obtg \obtd \obtf \btp \obte \obtf \btp \obtd $,\\
$ \btd \obtg \obtq \obtf \btt \obte \obtf \btp \obtd $,\\
$ \btd \obtp \btf \obtd \btt \btp \btd \btf \obtd $,\\ ~\\
$ \btd \btg \btp \btd \obtf \btr \btb \obti \obtd $,\\
$ \btd \btg \btq \btd \obtp \obtt \btd \btf \obtd $,\\
$ \btd \bti \obtb \obtu \btm \bth \obts \btg \obtd $,\\
$ \btd \bti \btk \obtj \obtk \obtq \obtp \btf \obtd $,\\
$ \btd \obtg \obtd \obtf \btp \btq \btd \btf \obtd $,\\
$ \btd \obtg \obtq \obtf \btt \btq \btd \btf \obtd $,\\
$ \btd \obtp \btf \obtd \btt \btq \btd \btf \obtd $,\\ ~\\
$ \btd \btg \btp \bte \obtg \obtd \obtq \obtg \obtd $,\\
$ \btd \btg \btq \btd \obtp \obtt \btd \obtg \obtd $,\\
$ \btd \bti \obtb \obtu \btm \bti \btd \btf \obtd $,\\
$ \btd \bti \btk \obtj \obtk \obto \obtl \obti \obtd $,\\
$ \btd \obtg \obtd \obtf \btp \obtd \obtp \btf \obtd $,\\
$ \btd \obtg \obtq \obtn \obtg \obtd \obtp \btf \obtd $,\\
$ \btd \obtp \btf \obtd \obtf \obtu \btn \btp \obtd $,\\ ~\\
$ \btd \btg \btp \bte \obtp \btf \btd \btf \obtd $,\\
$ \btd \btg \obtd \btp \btd \btf \btb \btp \obtd $,\\
$ \btd \bti \obtb \obtu \btn \obte \obtp \btf \obtd $,\\
$ \btd \bti \btk \obtj \obto \btl \btn \btp \obtd $,\\
$ \btd \obtg \obtd \obtf \btp \obtd \btt \btp \obtd $,\\
$ \btd \obtg \obtq \obtn \obtg \obtq \obtp \btf \obtd $,\\
$ \btd \obtp \obtn \obtg \obtd \obtq \obtp \btf \obtd $,\\ ~\\
$ \btd \btg \btp \bte \obtp \btf \btd \btg \obtd $,\\
$ \btd \btg \obtd \btp \btd \btf \btd \obtg \obtd $,\\
$ \btd \bti \obtb \obtu \btn \obte \btt \btp \obtd $,\\
$ \btd \bti \btk \obtj \obtc \btt \btq \btg \obtd $,\\
$ \btd \obtg \obtd \obtf \obtd \obtq \obtp \btf \obtd $,\\
$ \btd \obtg \obtq \obtn \obtl \bto \btn \btp \obtd $,\\
$ \btd \obtp \obtn \obtp \btf \btp \btd \btf \obtd $,\\ ~\\
$ \btd \btg \btp \bte \obtp \btf \btd \obtg \obtd $,\\
$ \btd \btg \obtd \btp \btd \btf \btq \btg \obtd $,\\
$ \btd \bti \obtb \obtu \btn \obte \obtf \btp \obtd $,\\
$ \btd \bti \btk \obtj \obtc \obtf \obtp \btf \obtd $,\\
$ \btd \obtg \obtd \obtf \obtd \obtp \obtf \btp \obtd $,\\
$ \btd \obtg \obtq \obtn \obtl \obtq \obtp \btf \obtd $,\\
$ \btd \obtp \obtn \obtp \btf \btq \btd \btf \obtd$,\\ ~\\
$ \btd \btg \btp \bte \obtp \btf \btq \btg \obtd $,\\
$ \btd \btg \obtd \btp \btd \btg \btn \btp \obtd $,\\
$ \btd \bti \obtb \obtu \btn \obta \obts \btg \obtd $,\\
$ \btd \bti \btk \obto \btk \bth \obts \btg \obtd $,\\
$ \btd \obtg \obtd \obtf \obtd \btt \btq \btg \obtd $,\\
$ \btd \obtg \obtq \obtn \obto \btl \btn \btp \obtd $,\\
$ \btd \obtp \obtn \obtp \btf \obtd \btt \btp \obtd $,\\ ~\\
$ \btd \btg \btp \bte \obtp \btf \obtd \obti \obtd $,\\
$ \btd \btg \obtd \btp \btd \btg \btq \btg \obtd $,\\
$ \btd \bti \obtb \obtu \btn \bto \obtk \obti \obtd $,\\
$ \btd \bti \btk \obto \btk \bti \btd \btf \obtd $,\\
$ \btd \obtg \obtd \obtf \btr \bte \obtp \btf \obtd $,\\
$ \btd \obtg \obtq \obtn \btu \bte \obtp \btf \obtd $,\\
$ \btd \obtp \obtn \obtp \obtf \btp \btd \btf \obtd $,\\ ~\\
$ \btd \btg \btp \bte \obtp \obtf \obtp \btf \obtd $,\\
$ \btd \btg \obtd \btp \btd \btg \obtd \obti \obtd $,\\
$ \btd \bti \obtb \obtu \btn \btp \btd \btf \obtd $,\\
$ \btd \bti \btk \obto \btl \obta \obts \btg \obtd $,\\
$ \btd \obtg \obtd \obtf \btr \btf \btd \btf \obtd $,\\
$ \btd \obtg \obtq \obtn \btu \btf \btb \btp \obtd $,\\
$ \btd \obtp \obtn \obtp \obtf \obtd \obtp \btf \obtd $,\\ ~\\
$ \btd \btg \btp \bte \obtp \obtf \obtp \obtg \obtd $,\\
$ \btd \btg \obtd \btr \bta \obtt \btb \btp \obtd $,\\
$ \btd \bti \obtb \obtu \btn \btp \btd \btg \obtd $,\\
$ \btd \bti \btk \obto \btl \bto \obtk \obti \obtd $,\\
$ \btd \obtg \obtd \obtf \btr \btf \btd \btg \obtd $,\\
$ \btd \obtg \obtq \obtn \btu \btf \btd \btf \obtd $,\\
$ \btd \obtp \obtn \obtp \obtf \obtd \btt \btp \obtd $,\\ ~\\
$ \btd \btg \btp \bte \obtp \obtf \btt \btp \obtd $,\\
$ \btd \btg \obtd \btr \bta \obtj \obtk \obti \obtd $,\\
$ \btd \bti \obtb \obtu \btn \btp \obtd \obti \obtd $,\\
$ \btd \bti \btk \obto \obtn \obtp \obtf \btp \obtd $,\\
$ \btd \obtg \obtd \obtf \obtc \btt \btq \btg \obtd $,\\
$ \btd \obtg \obtq \obtt \btb \btp \btd \btf \obtd $,\\
$ \btd \obtp \obtn \obtl \obtm \obtt \btd \btf \obtd$,\\ ~\\
$ \btd \btg \btp \bte \obtp \obtn \obtl \obti \obtd $,\\
$ \btd \btg \obtd \btr \bta \btu \btb \obti \obtd $,\\
$ \btd \bti \obtb \obtu \obtk \bto \obtk \obti \obtd $,\\
$ \btd \bti \btk \obto \obtn \btu \btb \btp \obtd $,\\
$ \btd \obtg \obtd \obtf \obtc \obtf \obtp \btf \obtd $,\\
$ \btd \obtg \obtq \obtt \btb \obtq \obtf \btp \obtd $,\\
$ \btd \obtp \obtn \obto \btk \bti \btd \btf \obtd$,\\ ~\\
$ \btd \btg \btp \bte \obtp \obtt \btb \btp \obtd $,\\
$ \btd \btg \obtd \btr \btb \btp \btd \btg \obtd $,\\
$ \btd \bti \obtb \obtu \obtk \obto \obtl \obti \obtd $,\\
$ \btd \bti \btk \obto \obtl \bto \btn \btp \obtd $,\\
$ \btd \obtg \obtd \obtf \obtu \btn \btq \btg \obtd $,\\
$ \btd \obtg \obtq \obtt \btc \bte \obtp \btf \obtd $,\\
$ \btd \obtp \obtn \btu \btb \btp \btd \btf \obtd$,\\ ~\\


$ \btd \btg \btp \bte \obtp \obtt \btb \obti \obtd $,\\
$ \btd \btg \obtd \btr \btb \obtd \obtp \btf \obtd $,\\
$ \btd \bti \obtb \obtr \btd \obtf \obtp \btf \obtd $,\\
$ \btd \bti \btk \obto \obtl \obtk \obts \btg \obtd $,\\
$ \btd \obtg \obtd \obtf \obtr \btf \btd \btf \obtd $,\\
$ \btd \obtg \obtq \obtt \btc \btf \btb \btp \obtd $,\\
$ \btd \obtp \obtn \btu \btb \obtd \obtp \btf \obtd $,\\ ~\\
$ \btd \btg \btp \bte \obtp \obtt \btd \btf \obtd $,\\
$ \btd \btg \obtd \btr \btb \obti \obtk \obti \obtd $,\\
$ \btd \bti \obtb \obtr \btf \btp \btd \btf \obtd $,\\
$ \btd \bti \btk \obto \obtl \obtq \obtp \btf \obtd $,\\
$ \btd \obtg \obtd \obtf \obtr \btf \btq \btg \obtd $,\\
$ \btd \obtg \obtq \obtt \btc \btf \btd \btf \obtd $,\\
$ \btd \obtp \obtn \btu \btb \obtd \btt \btp \obtd $,\\ ~\\
$ \btd \btg \btp \bte \obtp \obtt \btd \obtg \obtd $,\\
$ \btd \btg \obtd \btr \btb \obtq \obtf \btp \obtd $,\\
$ \btd \bti \obtb \obtr \btf \btp \btd \btg \obtd $,\\
$ \btd \bti \btl \btn \btp \obte \btt \btp \obtd $,\\
$ \btd \obtg \obtd \obtf \obtr \bth \obts \btg \obtd $,\\
$ \btd \obtg \obtq \obtt \btd \btf \btb \btp \obtd $,\\
$ \btd \obtp \obtn \btu \btf \btd \obtp \btf \obtd $,\\ ~\\
$ \btd \btg \btp \bte \obtn \btu \btb \btp \obtd $,\\
$ \btd \btg \obtd \btr \btc \btj \obtk \obti \obtd $,\\
$ \btd \bti \obtb \obtr \btf \btq \btd \btf \obtd $,\\
$ \btd \bti \btl \btn \btp \obte \obtf \btp \obtd $,\\
$ \btd \obtg \bts \bta \obtn \obtp \obtf \btp \obtd $,\\
$ \btd \obtg \obtq \obtt \btf \btq \btd \btf \obtd $,\\
$ \btd \obtp \obtt \btd \obtb \obtf \obtp \btf \obtd $,\\ ~\\
$ \btd \btg \btp \bte \obtn \btu \btb \obti \obtd $,\\
$ \btd \btg \obtd \btr \btc \obth \obtk \obti \obtd $,\\
$ \btd \bti \obtb \obtr \btf \obtd \btt \btp \obtd $,\\
$ \btd \bti \btl \btn \btp \obtd \obtp \btf \obtd $,\\
$ \btd \obtg \bts \bta \obtn \btu \btb \btp \obtd $,\\
$ \btd \obtp \obtb \btt \btm \btl \btn \btp \obtd $,\\
$ \btd \obtp \obtt \btd \btf \btq \btd \btf \obtd $,\\ ~\\
$ \btd \btg \btp \bte \obtt \btb \obtq \obtg \obtd $,\\
$ \btd \btg \obtd \obtq \obtg \obtd \obtp \btf \obtd $,\\
$ \btd \bti \obtb \obtr \btg \btp \btd \btf \obtd $,\\
$ \btd \bti \btl \obte \obtp \obtt \btb \btp \obtd $,\\
$ \btd \obtg \bts \bta \obtt \bte \obtp \btf \obtd $,\\
$ \btd \obtp \obtb \btt \obta \obtu \btn \btp \obtd $,\\
$ \btd \obtp \obtt \btd \btf \obtd \btt \btp \obtd $,\\ ~\\
$ \btd \btg \btp \bte \obtt \btd \obtp \btf \obtd $,\\
$ \btd \btg \obtd \obtq \obtg \btr \btb \btp \obtd $,\\
$ \btd \bti \obtb \obtr \btg \btp \btd \btg \obtd $,\\
$ \btd \bti \btl \obte \obtc \btt \btq \btg \obtd $,\\
$ \btd \obtg \bts \bta \obtt \btf \btq \btg \obtd $,\\
$ \btd \obtp \obtb \btt \btp \obte \btt \btp \obtd $,\\
$ \btd \obtp \obtt \btd \obtg \obtd \obtp \btf \obtd $,\\ ~\\
$ \btd \btg \btp \bte \obtt \btd \obtf \btp \obtd $,\\
$ \btd \btg \obtd \obtq \obtg \btr \btb \obti \obtd $,\\
$ \btd \bti \obtb \obtr \btg \btq \btd \btf \obtd $,\\
$ \btd \bti \btl \obte \obtc \obtf \obtp \btf \obtd $,\\
$ \btd \obtg \bts \bta \obtl \obtk \obts \btg \obtd $,\\
$ \btd \obtp \obtb \btt \btp \obte \obtf \btp \obtd $,\\
$ \btd \obtp \obtt \btd \obtg \obtq \obtp \btf \obtd $,\\ ~\\
$ \btd \btg \btp \bte \obtt \bte \obtp \btf \obtd $,\\
$ \btd \btg \obtd \obtq \obtg \obti \obtk \obti \obtd $,\\
$ \btd \bti \obtb \obtr \btg \btq \btd \btg \obtd $,\\
$ \btd \bti \btl \obte \obtu \btn \btq \btg \obtd $,\\
$ \btd \obtg \bts \bta \obtl \obtq \obtp \btf \obtd $,\\
$ \btd \obtp \obtb \btt \btp \obtd \obtp \btf \obtd $,\\
$ \btd \obtp \obtt \btd \obtp \btf \btd \btf \obtd $,\\ ~\\
$ \btd \btg \btp \bte \obtt \btf \btq \btg \obtd $,\\
$ \btd \btg \obtd \obtq \obtg \obtq \obtf \btp \obtd $,\\
$ \btd \bti \obtb \obtr \bth \obtu \btn \btp \obtd $,\\
$ \btd \bti \btl \bto \btk \bth \obts \btg \obtd $,\\
$ \btd \obtg \bts \obth \obtm \obtp \obtf \btp \obtd $,\\
$ \btd \obtp \obtb \obtf \btp \btq \btd \btf \obtd $,\\
$ \btd \obtp \obtt \btd \obtf \btp \btd \btf \obtd $,\\ ~\\
$ \btd \btg \btp \bte \obtl \obtk \obts \btg \obtd $,\\
$ \btd \btg \obtd \obtq \obtp \btd \obtf \btp \obtd $,\\
$ \btd \bti \btd \obtg \obtd \btp \btq \btg \obtd $,\\
$ \btd \bti \btl \bto \btk \bti \btd \btf \obtd $,\\
$ \btd \obtg \bts \obth \obtm \obtt \btd \btf \obtd $,\\
$ \btd \obtp \obtb \obtf \btp \obtd \obtp \btf \obtd $,\\
$ \btd \obtp \obtt \btd \obtf \obtd \obtp \btf \obtd $,\\ ~\\
$ \btd \btg \btp \bte \obtl \obtq \obtp \btf \obtd $,\\
$ \btd \btg \obtd \obtq \obtp \btf \btb \btp \obtd $,\\
$ \btd \bti \btd \obtg \obtd \obtq \obtp \btf \obtd $,\\
$ \btd \bti \btl \bto \btm \bth \obts \btg \obtd $,\\
$ \btd \obtg \bts \obth \obtm \btu \btb \btp \obtd $,\\
$ \btd \obtp \obtb \obtf \btp \obtd \btt \btp \obtd $,\\
$ \btd \obtp \obtt \bte \obtp \btf \btd \btf \obtd $,\\ ~\\
$ \btd \btg \btp \btf \btc \bte \obtp \btf \obtd $,\\
$ \btd \btg \obtd \obtq \obtp \btf \btd \btf \obtd $,\\
$ \btd \bti \btd \obtg \obtd \obtq \obtf \btp \obtd $,\\
$ \btd \bti \btl \bto \btm \btl \btn \btp \obtd $,\\
$ \btd \obtg \bts \obth \obtk \bto \btn \btp \obtd $,\\
$ \btd \obtp \obtb \obtf \obtd \obtq \obtp \btf \obtd $,\\
$ \btd \obtp \obtt \bte \obtp \obtf \obtp \btf \obtd $,\\ ~\\
$ \btd \btg \btp \btf \btc \btf \btd \btf \obtd $,\\
$ \btd \btg \obtd \obtq \obtp \btf \btd \btg \obtd $,\\
$ \btd \bti \btd \obtg \obtd \obtp \obtf \btp \obtd $,\\
$ \btd \bti \btl \bto \btn \btp \btd \btg \obtd $,\\
$ \btd \obtg \bts \obth \obtk \obtq \obtp \btf \obtd $,\\
$ \btd \obtp \obtb \obtf \obtd \obtq \obtf \btp \obtd $,\\
$ \btd \obtp \obtt \bte \obtp \obtt \btd \btf \obtd $,\\ ~\\
$ \btd \btg \btp \btf \btc \btf \btd \btg \obtd $,\\
$ \btd \btg \obtd \obtq \obtp \btf \btd \obtg \obtd $,\\
$ \btd \bti \btd \obtg \btr \btf \btd \btf \obtd $,\\
$ \btd \obtg \obtd \obtq \obtg \obtd \obtp \btf \obtd $,\\
$ \btd \obtg \bts \obth \obtd \btp \btd \btf \obtd $,\\
$ \btd \obtp \obtb \obtf \obtd \obtp \obtf \btp \obtd $,\\
$ \btd \obtf \btp \btd \obtb \obtf \obtp \btf \obtd $,\\ ~\\
$ \btd \btg \btp \btf \btc \btf \btd \obtg \obtd $,\\
$ \btd \btg \obtd \obtq \obtp \btf \obtd \obti \obtd $,\\
$ \btd \bti \btd \obtp \btf \btp \btd \btf \obtd $,\\
$ \btd \obtg \obtd \obtq \obtg \btr \btb \btp \obtd $,\\
$ \btd \obtg \bts \obth \obtd \btp \btq \btg \obtd $,\\
$ \btd \obtp \obtb \obtf \btr \btf \btb \btp \obtd $,\\
$ \btd \obtf \btp \btd \btf \btq \btd \btf \obtd $,\\ ~\\
$ \btd \btg \btp \btf \btc \btj \obtk \obti \obtd $,\\
$ \btd \btg \obtd \obtq \obtf \btp \btd \btf \obtd $,\\
$ \btd \bti \btd \obtp \btf \btq \btd \btf \obtd $,\\
$ \btd \obtg \obtd \obtq \obtg \obtq \obtf \btp \obtd $,\\
$ \btd \obtg \bts \obth \obtd \obtq \obtp \btf \obtd $,\\
$ \btd \obtp \obtb \obtf \obtr \btd \obtf \btp \obtd $,\\
$ \btd \obtf \btp \btd \btg \btp \btd \btf \obtd $,\\ ~\\
$ \btd \btg \btp \btf \btp \obte \btt \btp \obtd $,\\
$ \btd \btg \obtd \obtq \obtf \btp \btd \btg \obtd $,\\
$ \btd \bti \btd \obtp \btf \btq \btd \btg \obtd $,\\
$ \btd \obtg \obtd \obtq \obtp \btd \obtf \btp \obtd $,\\
$ \btd \obtg \bts \obth \obtd \obtq \obtf \btp \obtd $,\\
$ \btd \obtp \obtb \obtu \btm \bti \btd \btf \obtd $,\\
$ \btd \obtf \btp \btd \btg \btq \btd \btf \obtd $,\\ ~\\
$ \btd \btg \btp \btf \btp \obte \obtf \btp \obtd $,\\
$ \btd \btg \obtd \obtq \obtf \btp \btn \btp \obtd $,\\
$ \btd \bti \btd \obtp \obtn \obtp \obtf \btp \obtd $,\\
$ \btd \obtg \obtd \obtq \obtp \btf \btb \btp \obtd $,\\
$ \btd \obtg \bts \obth \btr \bte \obtp \btf \obtd $,\\
$ \btd \obtp \obtb \obtu \btm \btl \btn \btp \obtd $,\\
$ \btd \obtf \btp \btd \obtg \obtq \obtp \btf \obtd $,\\ ~\\
$ \btd \btg \btp \btf \btq \bth \obts \btg \obtd $,\\
$ \btd \btg \obtd \obtq \obtf \btp \btq \btg \obtd $,\\
$ \btd \bti \btd \obtp \obtn \btu \btb \btp \obtd $,\\
$ \btd \obtg \obtd \obtq \obtp \btf \btd \btf \obtd $,\\
$ \btd \obtg \bts \obth \btr \btf \btb \btp \obtd $,\\
$ \btd \obtp \obtb \obtu \btn \obte \obtp \btf \obtd $,\\
$ \btd \obtf \btp \btd \obtp \btf \btd \btf \obtd $,\\ ~\\
$ \btd \btg \btp \btf \obtd \btr \btb \btp \obtd $,\\
$ \btd \btg \obtd \obtq \obtf \btp \obtd \obti \obtd $,\\
$ \btd \bti \btd \obtp \obtt \btd \obtf \btp \obtd $,\\
$ \btd \obtg \obtd \obtq \obtp \btf \btd \btg \obtd $,\\
$ \btd \obtg \bts \obth \btr \btf \btd \btf \obtd $,\\
$ \btd \obtp \obtb \obtu \btn \obte \obtf \btp \obtd $,\\
$ \btd \obtf \btp \btd \obtp \obtt \btd \btf \obtd $,\\ ~\\
$ \btd \btg \btp \btf \obtd \obtq \obtf \btp \obtd $,\\
$ \btd \btg \obtd \obtq \obtf \obtd \obtp \btf \obtd $,\\
$ \btd \bti \btd \obtp \obtt \bte \obtp \btf \obtd $,\\
$ \btd \obtg \obtd \obtq \obtf \btp \btd \btf \obtd $,\\
$ \btd \obtg \bts \obth \obtq \obtf \obtp \btf \obtd $,\\
$ \btd \obtp \obtb \obtu \btn \btp \btd \btf \obtd $,\\
$ \btd \obtf \btp \btf \btb \btp \btd \btf \obtd $,\\ ~\\
$ \btd \btg \btp \btf \obtd \obtp \btd \obtg \obtd $,\\
$ \btd \btg \obtd \obtq \obtf \obtd \btt \btp \obtd $,\\
$ \btd \bti \btd \obtp \obtt \btf \btq \btg \obtd $,\\
$ \btd \obtg \obtd \obtq \obtf \btp \btd \btg \obtd $,\\
$ \btd \obtg \bts \obth \obtq \obtt \btd \btf \obtd $,\\
$ \btd \obtp \obtb \obtr \btd \obtf \obtp \btf \obtd $,\\
$ \btd \obtf \btp \btf \btc \btf \btd \btf \obtd $,\\ ~\\
$ \btd \btg \btp \btf \obtd \obtp \obtf \btp \obtd $,\\
$ \btd \btg \obtd \obtq \obtf \btr \btb \btp \obtd $,\\
$ \btd \bti \btd \obtf \btp \obte \btt \btp \obtd $,\\
$ \btd \obtg \obtd \obtq \obtf \btp \btn \btp \obtd $,\\
$ \btd \obtg \bts \btk \btj \obtk \obts \btg \obtd $,\\
$ \btd \obtp \obtb \obtr \btf \btp \btd \btf \obtd $,\\
$ \btd \obtf \btp \obte \btt \btp \btd \btf \obtd $,\\ ~\\
$ \btd \btg \btq \obtb \btt \obte \obtp \btf \obtd $,\\
$ \btd \btg \obtd \obtq \obtf \btr \btb \obti \obtd $,\\
$ \btd \bti \btd \obtf \btp \obte \obtf \btp \obtd $,\\
$ \btd \obtg \obtd \obtq \obtf \btp \btq \btg \obtd $,\\
$ \btd \obtg \bts \btk \btl \btn \btq \btg \obtd $,\\
$ \btd \obtp \obtb \obtr \btf \btp \btn \btp \obtd $,\\
$ \btd \obtf \btp \obte \obtr \btf \btd \btf \obtd $,\\ ~\\
$ \btd \btg \btq \obtb \btt \obte \btt \btp \obtd $,\\
$ \btd \btg \obtd \obtq \obtf \btt \btq \btg \obtd $,\\
$ \btd \bti \btd \obtf \btp \btq \btd \btf \obtd $,\\
$ \btd \obtg \obtd \obtq \obtf \obtd \obtp \btf \obtd $,\\
$ \btd \obtg \bts \btk \btl \obte \obtp \btf \obtd $,\\
$ \btd \obtp \obtb \obtr \btf \btq \btd \btf \obtd $,\\
$ \btd \obtf \btp \btq \btg \btp \btd \btf \obtd $,\\ ~\\
$ \btd \btg \btq \obtb \btt \obte \obtf \btp \obtd $,\\
$ \btd \btg \obtd \obtq \obtt \btd \obtp \btf \obtd $,\\
$ \btd \bti \btd \obtf \btp \btq \btd \btg \obtd $,\\
$ \btd \obtg \obtd \obtq \obtf \obtd \btt \btp \obtd $,\\
$ \btd \obtg \bts \btk \btl \obte \btt \btp \obtd $,\\
$ \btd \obtp \obtb \obtr \btf \obtd \btt \btp \obtd $,\\
$ \btd \obtf \btp \obtd \btt \btp \btd \btf \obtd $,\\ ~\\
$ \btd \btg \btq \obtb \btt \obta \obts \btg \obtd $,\\
$ \btd \btg \obtd \obtq \obtt \btd \obtf \btp \obtd $,\\
$ \btd \bti \btd \obtf \btp \obtd \obtp \btf \obtd $,\\
$ \btd \obtg \obtd \obtq \obtf \btr \btb \btp \obtd $,\\
$ \btd \obtg \bts \btk \btl \obte \obtf \btp \obtd $,\\
$ \btd \obtp \obtb \obtr \btg \btp \btd \btf \obtd $,\\
$ \btd \obtf \btp \obtd \btt \btq \btd \btf \obtd $,\\ ~\\
$ \btd \btg \btq \obtb \btt \bto \obtk \obti \obtd $,\\
$ \btd \btg \obtd \obtq \obtt \btf \btq \btg \obtd $,\\
$ \btd \bti \btd \obtf \btp \obtd \btt \btp \obtd $,\\
$ \btd \obtg \obtd \obtq \obtf \btt \btq \btg \obtd $,\\
$ \btd \obtg \bts \btk \obtj \bto \btn \btp \obtd $,\\
$ \btd \obtp \obtb \obtr \btg \btq \btd \btf \obtd $,\\
$ \btd \obtf \obtd \obti \obtm \obtt \btd \btf \obtd $,\\ ~\\
$ \btd \btg \btq \obtb \btt \btp \obtd \obti \obtd $,\\
$ \btd \btg \obtd \btt \btm \bth \obts \btg \obtd $,\\
$ \btd \bti \btd \obtf \obtr \btg \btq \btg \obtd $,\\
$ \btd \obtg \obtd \obtq \obtt \btd \obtp \btf \obtd $,\\
$ \btd \obtg \bts \btk \obto \btl \btn \btp \obtd $,\\
$ \btd \obtp \obtb \obtr \bth \obtu \btn \btp \obtd $,\\
$ \btd \obtf \obtd \obtq \obtp \btf \btd \btf \obtd $,\\ ~\\
$ \btd \btg \btq \obtb \btt \btq \btd \obtg \obtd $,\\
$ \btd \btg \obtd \btt \btm \bti \btd \btf \obtd $,\\
$ \btd \bti \btk \bth \obto \btj \obtk \obti \obtd $,\\
$ \btd \obtg \obtd \obtq \obtt \btd \obtf \btp \obtd $,\\
$ \btd \obtg \obtq \obtg \obtd \btp \btq \btg \obtd $,\\
$ \btd \obtp \btf \bte \obtp \btf \btd \btf \obtd $,\\
$ \btd \obtf \obtd \obtq \obtf \btp \btd \btf \obtd $,\\ ~\\
$ \btd \btg \btq \obtb \obtf \obtp \btd \obtg \obtd $,\\
$ \btd \btg \obtd \btt \btm \btj \obtk \obti \obtd $,\\
$ \btd \bti \btk \bth \obto \btl \btn \btp \obtd $,\\
$ \btd \obtg \obtd \obtq \obtt \btf \btq \btg \obtd $,\\
$ \btd \obtg \obtq \obtg \obtd \obtq \obtp \btf \obtd $,\\
$ \btd \obtp \btf \bte \obtp \obtf \obtp \btf \obtd $,\\
$ \btd \obtf \obtd \obtp \btf \btp \btd \btf \obtd $,\\ ~\\
$ \btd \btg \btq \obtb \obts \btg \btn \btp \obtd $,\\
$ \btd \btg \obtd \btt \btm \btl \btn \btp \obtd $,\\
$ \btd \bti \btk \bti \obtb \btt \btq \btg \obtd $,\\
$ \btd \obtg \obtd \obtp \obtb \btt \btq \btg \obtd $,\\
$ \btd \obtg \obtq \obtg \obtd \obtq \obtf \btp \obtd $,\\
$ \btd \obtp \btf \bte \obtp \obtt \btd \btf \obtd $,\\
$ \btd \obtf \obtd \obtf \btr \btf \btd \btf \obtd$.\\ ~\\
\end{quote}
\end{proof}

\section{Code availability}
\label{Sec:CodeAvailability}

The $54$-Brinkhuis $952$-triple can be verified using the code and script
to run it found at \cite{DouglasURL}.
The code is general and can be used to find special Brinkhuis triples for
general values of $n$.
While the code is fast for small values of $n$, e.g., $n \leq 35$, it will
take a {\em very} long time for $n=54$.
In addition, the output files will take Terabytes of disk space.

A sample build and run of the code for $n=35$ is the following:
\begin{quote}
\% script log\\
Script started on Sat Apr 16 19:20:23 2016\\
\% make runit N=35\\
clang -w -O3 brinkhuis.c -o brinkhuis\\
clang -w -O3 brinkhuis2t1.c -o brinkhuis2t1\\
clang -w -O3 brinkhuis2t2.c -o brinkhuis2t2\\
./doit 35\\
Success [  1]: 01202102010212010201202120102120210\\
Success [  2]: 01202102010212010201210120102120210\\
Success [  3]: 01202102010212012101202120102120210\\
Success [  4]: 01202102010212021012021201020120210 (palindromic)\\
Success [  5]: 01202102012101201020121020102120210\\
Success [  6]: 01202102012101202120102012102120210\\
Success [  7]: 01202102012101202120121020102120210\\
Success [  8]: 01202102012102120210121020102120210\\
Success [  9]: 01202120102012101201021012102120210\\
Success [ 10]: 01202120102012101202102012102120210\\
Success [ 11]: 01202120102012102010210120102120210\\
Success [ 12]: 01202120102012102010212012102120210\\
Success [ 13]: 01202120102012102120121020102120210 (palindromic)\\
Success [ 14]: 01202120102101210201210120102120210 (palindromic)\\
Success [ 15]: 01202120102120210120212012102120210\\
Success [ 16]: 01202120102120210201021012102120210\\
Success [ 17]: 01202120121020102120102012102120210 (palindromic)\\
Success [ 18]: 01210201021012010212021012010201210\\
Success [ 19]: 01210201021012010212021012021201210\\
Success [ 20]: 01210201021012021020120212010201210\\
Success [ 21]: 01210201021012021201020121021201210\\
Success [ 22]: 01210201021012021201210212010201210\\
Success [ 23]: 01210201021012102120102012021201210\\
Success [ 24]: 01210201021012102120121012010201210 (palindromic)\\
Success [ 25]: 01210201021012102120210201021201210\\
Success [ 26]: 01210201021201020120210121021201210\\
Success [ 27]: 01210201021201020120210201021201210\\
Success [ 28]: 01210201021201210120102012021201210\\
Success [ 29]: 01210201021201210212021012021201210\\
Success [ 30]: 01210201021202101201020121021201210\\
Success [ 31]: 01210201021202101201021012021201210\\
Success [ 32]: 01210201021202102010210121021201210\\
Success [ 33]: 01210201021202102012101201021201210\\
Success [ 34]: 01210212010201202101210201021201210\\
Success [ 35]: 01210212010201210120102012021201210\\
Success [ 36]: 01210212010201210120210121021201210\\
Success [ 37]: 01210212012101202120102012021201210\\
Success [ 38]: 01210212012102010212021012021201210\\
Success [ 39]: 01210212012102012021020121021201210 (palindromic)\\
Success [ 40]: 01210212021020102120102012021201210 (palindromic)\\
Done:  a1=109, a1p=  9, a1n= 50;   b1= 30, b1p=  4, b1n= 13\\
       a2=142, a2p=  6, a2n= 68;   b2= 43, b2p=  3, b2n= 20\\
Generated a1 and a2 files.\\
17 admissible words of length 35 read in\\
admissible triples:\\
328 admissible triples found\\
Generated t1.\\
23 admissible words of length 35 read in\\
admissible triples:\\
483 admissible triples found\\
Generated t2.\\
\% exit\\
exit\\
 ~\\
Script done on Sat Apr 16 19:20:39 2016\\
\end{quote}

\noindent
2000 {\em Mathematics Subject Classifications}:
Primary 57M15;
Secondary 11Y55

\noindent
(Concerned with sequence \seqnum{A006156})

\end{document}